\documentclass[12pt]{amsart}

\usepackage[utf8]{inputenc}
\usepackage{bm}
\usepackage{color}
\usepackage{a4wide}
\usepackage{amsmath,amssymb,amscd,amsthm,amsfonts}
\usepackage{nameref, hyperref}
\usepackage{kotex}
\usepackage{esint}
\usepackage{indentfirst}
\usepackage{graphicx}
\usepackage{mathtools}
\usepackage{mathabx}
\usepackage{bbm}
\usepackage{setspace}
\usepackage{tikz}
\usepackage{array}
\usepackage{textcomp}
\usepackage{enumitem}

\setlist{
leftmargin=*,
label=(\arabic*)
}

\newlist{remlist}{enumerate}{1}
\setlist[remlist]{label=(\arabic{remlisti}), ref=\thedefinition.(\arabic{remlisti}),noitemsep}

\newlist{thmlist}{enumerate}{1}
\setlist[thmlist]{label=(\arabic{thmlisti}), ref=\thedefinition.(\arabic{thmlisti}),noitemsep}

\usepackage[capitalize]{cleveref}
\usepackage{letltxmacro}
\usepackage{thmtools}

\newcommand{\QQ}{\mathbf{Q}}
\newcommand{\RR}{\mathbf{R}}
\newcommand{\CC}{\mathbf{C}}
\newcommand{\ZZ}{\mathbf{Z}}
\newcommand{\mcH}{\mathcal{H}}
\newcommand{\mcV}{\mathcal{V}}

\newcommand*{\medcap}{\mathbin{\scalebox{1.5}{\ensuremath{\cap}}}}%

\newcommand{\abs}[1]{\lvert #1 \rvert}
\newcommand{\inp}[1]{\langle #1 \rangle}
\newcommand{\bigo}{\mathrm{O}}
\newcommand{\restr}[2]{\ensuremath{\left.#1\right|_{#2}}}

\DeclareMathOperator{\interior}{int}
\DeclareMathOperator{\dom}{dom}
\DeclareMathOperator{\epi}{epi}
\DeclareMathOperator{\conv}{conv}
\DeclareMathOperator{\cone}{cone}

\theoremstyle{plain}

\newtheorem{definition}{Definition}[section]
\newtheorem{theorem}[definition]{Theorem}
\newtheorem{corollary}[definition]{Corollary}
\newtheorem{lemma}[definition]{Lemma}
\newtheorem{proposition}[definition]{Proposition}

\newtheorem{question}[definition]{Question}
\newtheorem{example}[definition]{Example}
\newtheorem{rmk}[definition]{Remark}

\begin{document}

\onehalfspacing

\title[On equisingular approximation of plurisubharmonic functions]{On  equisingular approximation   of \\ plurisubharmonic functions}



\author{Jongbong An and Hoseob Seo}

\begin{abstract}

It is a natural question to ask which plurisubharmonic functions admit a `nice' approximation in the sense of a decreasing equisingular approximation with analytic singularities. For arbitrary toric plurisubharmonic functions, we give a criterion  for admitting a nice approximation with toric approximants. Our results are motivated by a recent result of Guan for toric plurisubharmonic functions of the diagonal type. 
\end{abstract}

\maketitle \vspace{-2em}

\section{Introduction}

Plurisubharmonic (psh for short) functions are fundamental objects in several complex variables. They play important roles in an increasing number of applications in complex analysis and geometry, cf. \cite{D10}.

In the fundamental work \cite{D92}, Demailly gave a crucial method of approximating a general psh function $\varphi$ by ones easier to understand, namely $\varphi_m$ with analytic singularities  given by multiplier ideals $\mathcal{J} (m\varphi)$ (up to power $\frac{1}{m}$) for $m \ge 1$. This Demailly approximation of a psh function has had far-reaching developments and applications, see e.g. \cite{DK01}, \cite{DPS01}, \cite{D10}, \cite{D13}, \cite{R12}, \cite{K14},  \cite{G16}, \cite{G20} and others. 

In \cite[Theorem~2.3]{DPS01}, an important variant of Demailly approximation was established where one can approximate $\varphi$ by a decreasing {\bf equisingular} sequence $\varphi_m \to \varphi$ in the sense that the multiplier ideals are constant, i.e.  $\mathcal{J} (\varphi_m) = \mathcal{J} (\varphi)$ for every $m \ge 1$.  Such a decreasing equisingular approximation was applied in the  proof of the hard Lefschetz theorem \cite[Theorem~2.1]{DPS01}. 

 It was then natural to ask whether such a decreasing equisingular approximation can be taken so that the approximants $\varphi_m$ also  have analytic singularities as in the original Demailly approximation.
  Qi'an Guan  showed by an example \cite{G16} that in general, one cannot expect all three of `decreasing', `equisingular' and `analytic singularities' to hold simultaneously for approximations of psh functions. On the other hand, it is known (from \cite{D92}, \cite{DPS01} and \cite{D13}) that any two of the three can be made to hold in an approximation.

  One can then ask when a psh function $\varphi$ admits such a nice approximation, i.e. a decreasing equisingular approximation by psh functions with analytic singularities. For this question, it suffices to assume that $\varphi$ does not have analytic singularities, since otherwise one can use the constant sequence given by $\varphi$ itself.

In this direction,  for toric psh functions of the diagonal type, Qi'an Guan gave the following criterion for the existence of a nice approximation.

\begin{theorem}[Qi'an Guan] \cite[Theorem~1.1]{G20} \label{G20}
Let $1 \le m \le n$ be integers. Let $a_1, \ldots, a_m$ be positive real numbers. Let $\varphi = \log \sum\limits_{i=1}^m \abs{z_i}^{a_i}$ as a psh function on $\CC^n$. Suppose that $\varphi$ does not have analytic singularities. Then the following (1) and (2) are equivalent : 
\begin{thmlist}
\item  $\varphi$ admits a nice approximation, i.e. a  decreasing equisingular approximation by psh functions with analytic singularities near $0$.
\item The equation $\sum\limits_{i=1}^m \frac{x_i}{a_i} = 1$ has no positive integer solutions. \label{G20_2}
\end{thmlist}
\end{theorem}

The function $\varphi$ in Theorem~\ref{G20} can indeed have non-analytic singularities when $a_i$'s are irrational, see Example~\ref{root}. 
We also  note that the condition (2) in Theorem~\ref{G20} is equivalent to the following one in terms of jumping numbers of multiplier ideals :

{\it (2') The number $1$ is not a jumping number of $\varphi$ at $0$.}

This follows from  \cite[Proposition~3.3]{KS20} and  the facts that the set of jumping numbers of $\varphi$ of the form $\log{ \sum_{i=1}^m { \abs{z_i}^{a_i} } }$ is discrete  and 
 that the Newton convex body of $\varphi = \log{ \sum_{i=1}^m { \abs{z_i}^{a_i}}}$ is given by
\[
  P(\varphi) = \left\{ \sum_{i=1}^{m} t_ia_i e_i  : t_1, \ldots, t_m \ge 0, t_1 + \cdots + t_m = 1 \right\} + \RR^n_{\ge 0}.
\]

 In this paper, we have the following main result  which is a weaker version of Theorem~\ref{G20} for arbitrary toric psh functions: weaker in that the approximants are also toric. 

\begin{theorem} \label{mainthm1}
Let $\varphi$ be a toric psh function defined on $D(0,r) \subset \CC^n$. Then the following are equivalent.
\begin{thmlist}
\item $\varphi$ admits a nice toric approximation, i.e. a decreasing equisingular approximation by toric psh functions with analytic singularities.
\item There exists a polyhedron $P \subseteq \RR^n_{\ge 0}$ satisfying the following three conditions: \label{mainthm1_2}
\begin{enumerate}[label=(\roman*)]
    \item $(2/c)P$ is a rational polyhedron for some $c>0$,
    \item $P(\varphi) \subseteq P$ and $P + \RR^n_{\ge 0} \subseteq P$,
    \item $(\interior P) \cap \ZZ^n_{\ge 0} = (\interior P(\varphi)) \cap \ZZ^n_{\ge 0}$.
\end{enumerate}
\end{thmlist}
\end{theorem}

\noindent Here $\RR^n_{\ge 0}$(resp. $\ZZ^n_{\ge 0}$) is the set of $n$-tuples of nonnegative real numbers (resp. integers) and $P(\varphi)$ is the Newton convex body of $\varphi$ (see Definition~\ref{Newtonconvexbody} for the definition). 
Also $r = (r_1, \ldots , r_n)$ is a polyradius of a polydisk in $\CC^n$. A \textit{polyhedron} is a finite intersection of upper hyperplanes in $\RR^n$ (see Definition~\ref{H-poly}, Definition~\ref{V-poly} and Theorem~\ref{H_V_equivalence}). If all the equations of hyperplanes can be represented by rational coefficients and rational constants, we say that the polyhedron is rational.

We note that (as would be naturally expected) when $\varphi$ is of the form $\log{ \sum_{i=1}^m{ \abs{z_i}^{a_i} } }$, the condition  Theorem~\ref{mainthm1_2} implies the condition Theorem~\ref{G20_2}.
Indeed, suppose that  $\alpha = (\alpha_1, \ldots , \alpha_n)$ is a positive integer solution of $\sum\limits_{i=1}^m \frac{x_i}{a_i} = 1$. In this case, the Newton convex body of $\varphi$ is given by
$$
P(\varphi) = \left\{ \sum_{i=1}^{m} t_ia_i e_i  : t_1, \ldots, t_m \ge 0, t_1 + \cdots + t_m = 1 \right\} + \RR^n_{\ge 0}
$$
where $e_1,\ldots, e_n$ are the standard basis for $\RR^n$. Since the slope of the hyperplane $\sum_{i=1}^{n}{ \frac{x_i}{a_i} } = \frac{1}{c}$ cannot be rational for every $c > 0$, the interior of $P \cap \RR^n_{>0}$ contains $P(\varphi) \cap \RR^n_{>0}$. This implies that $\alpha$ belongs to $\interior P$, which contradicts the condition $(iii)$ of Theorem~\ref{mainthm1_2}.

For what we observed as in the condition (2'), the existence of a nice approximation seems to be guaranteed only when $(1-\epsilon)\varphi$ is equisingular to $\varphi$ for sufficiently small $\epsilon>0$. As an application of our main result, we are now able to consider the following natural question which arises from Theorem~\ref{G20}. 

\begin{question} \label{QQ}

Let $\varphi$ be a psh function. Does the following implication hold ? : 

If $\varphi$ admits a nice approximation (i.e. a decreasing equisingular approximation by psh functions with analytic singularities), then $1$ is not a jumping number of $\varphi$. 

\end{question}

We note that  this implication certainly holds for the psh functions of the diagonal type in  Theorem~\ref{G20}, in view of the above condition (2').  
As a consequence of Theorem~\ref{mainthm1}, we answer this question negatively by establishing the following result. 

\begin{corollary} \label{cor:always_approximated}
There exists a toric psh function $\varphi$ on $(\CC^n, 0)$ such that 
\begin{thmlist}
\item for some $a>0$, we have nontrivial multiplier ideal  $\mathcal{J}(a \varphi)_0 \neq \mathcal{O}_{\CC^n, 0}$,

\item for every $c>0$, the psh function $c\varphi$ admits a decreasing equisingular approximation by toric psh functions with analytic singularities. \label{cor:always_approximated_2}
\end{thmlist}
\end{corollary}

Indeed, let $\varphi$ be as in Corollary~\ref{cor:always_approximated}. Let $\lambda$ be the log canonical threshold of $\varphi$ at $0$ (which exists by the condition (1)), i.e. the smallest jumping number of $\varphi$ at $0$. Then $\mathcal{J}(\lambda^{}\varphi)_0$ is not trivial and $\lambda^{}\varphi$ admits a nice toric approximation, which answers Question~\ref{QQ} negatively.

Finally, we have some remarks on the methods of proof. 
In the course  of the proof of our main result Theorem~\ref{mainthm1}, we also establish some results in convex analysis. Perhaps the most interesting result among these is that if the epigraph of a convex function defined on $\RR^n_{\le 0}$ is a rational polyhedron, then so is the epigraph of its convex conjugate. This result plays an important role in the construction of equisingular approximants.

Our main strategy for the proof of Theorem~\ref{mainthm1} is to consider convex conjugates of toric psh functions. We present an explicit characterization for convex functions associated to toric psh functions with analytic singularities. Then we show the relation between convex functions and their conjugates when convex functions are from toric psh functions with analytic singularities. Using this we prove the main theorem using convergence of convex conjugates.

The paper is organized as follows. In Section 2, we give preliminary results on toric psh functions and convex analysis. In Section 3, we characterize how toric psh functions with analytic singularities and their Newton convex bodies look like. We also interpret the result of Guan \cite{G20} using convex analysis related to toric psh functions. In Section 4, we observe how convex conjugates of toric psh functions with analytic singularities should behave and demonstrate relationships between the convergence of convex functions and the convergence of their conjugates. In Section 5, we prove the main theorem and present some relevant examples.\\

\noindent \textbf{Acknowledgements.} We would like to thank Dano Kim for his valuable comments and his constant support. We also would like to thank Liran Rotem for helpful discussions. The second named author was supported by Basic Science Research Program through the National Research Foundation of Korea(NRF) funded by the Ministry of Education(2020R1A6A3A01099387).

\section{Preliminaries}

In this section, we prepare some preliminary results on toric psh functions and convex analysis which are mainly covered by the references  \cite{H},  \cite{Sch} and  \cite{Gu12}.  We first recall the definition of a plurisubharmonic function. 

\begin{definition} 

 Let $U$ be an open subset of $\CC^n$.  Let $\varphi : U \to \RR \cup \{ -\infty \}$ be a function which is not identically equal to $  -\infty $ on any connected component of $U$. The function $\varphi$   is called {\bf plurisubharmonic} (psh for short) if 
 
\begin{enumerate}
  \item $\varphi$ is upper semicontinuous, and 
  \item for any complex line $L$ in $\CC^n$, the restriction of $\varphi$ to $L \cap U$ is either  subharmonic or identically equal to $  -\infty $ (on each connected component of $L \cap U$). 
\end{enumerate}
\end{definition}
 
  We refer to \cite{DX}, \cite{D10}, \cite{Ki94} for  introduction to some general properties of psh functions.
For given holomorphic functions $g_1, \ldots, g_m$ on $U$, and $c \ge 0$, one can define a psh function $\varphi$ by
\begin{equation} \label{def:psh_function}
\varphi (z) := c \log{(\abs{g_1 (z)}^2 + \cdots + \abs{g_m(z)}^2)}.
\end{equation}
which gives an example of psh functions. Moreover, if a psh function $\varphi$ is locally of the form (\ref{def:psh_function}) up to $O(1)$, we will say that $\varphi$ has {\bf analytic singularities} following the terminology of \cite{D10}.

\begin{example}\label{root}

Note that the function $\varphi$ in Theorem~\ref{G20} need not have analytic singularities when $a_i$'s are irrational, cf. \cite[Example 4.1]{K16}. For example, $\varphi(z_1,z_2) := \log{(|z_1|^{\sqrt{2}} + |z_2|^{\sqrt{3}})}$ in $\CC^2$ does not have analytic singularities but satisfies the condition (2) in Theorem~\ref{G20}. Indeed, if $g$ is the associated convex function to $\varphi$ (see Proposition~\ref{Gueconvex}), then it can be easily shown that the epigraph of $g$ cannot be a rational polyhedron even if one restrict the domain of $g$ to be a rational polyhedron. In view of Lemma~\ref{epigraph_polyhedron_maxaffine}, $\varphi$ does not have analytic singularities. Furthermore, since every $\QQ$-linear combination of $\sqrt{2}$ and $\sqrt{3}$ cannot be a nonzero integer, $\varphi$ satiesfies the condition of Theorem~\ref{G20_2}. Therefore $\varphi$ has a nice approximation near $0$.

\end{example}

  Now we turn to toric psh functions.

\begin{definition}
Let $\varphi$ be a psh function defined on a  polydisk $D(0,r) \subset \CC^n$ with polyradius $(r, \ldots , r)$ where $r >0$. We will say $\varphi$ is  {\bf toric} if $\varphi$ is invariant under the torus action, i.e. $\varphi(z_1 , \ldots , z_n)= \varphi(e^{i\theta_1}z_1 , \ldots , e^{i\theta_n}z_n)$ where each $\theta_j$ is an arbitrary real number for $1 \le j \le n$.
\end{definition}

 See also e.g. \cite{R13}, \cite{Gu12}, \cite{KR18} for some more information on toric psh functions.
We recall the following result from \cite{DX}, cf. \cite[Proposition 1.11]{Gu12}. 

\begin{proposition}  \label{Gueconvex}
 Let $\varphi$ be a toric psh function defined on $D(0,r)$. Then there exists a convex function $g$ defined on $(-\infty, \log r)^n$, increasing in each variable, such that for all $z \in D(0,r)$, we have $\varphi(z)=g(\log \abs{z_1} , \ldots , \log \abs{z_n})$.
\end{proposition}

The papers \cite{R13},  \cite{Gu12}  proved independently a characterization of multiplier ideal for toric psh functions in the language of convex analysis. For this, we first define the Newton convex body related to a convex function.

 \begin{definition}
   Let $g : \RR^n \longrightarrow (-\infty, +\infty]$ be a convex function not identically $+\infty$. We define the {\bf convex conjugate} $g^{*} : \RR^n \longrightarrow (-\infty, +\infty]$ by $$
   g^{*}(x) := \displaystyle\sup_{y\in \RR^n} \left( \langle x,y \rangle -g(y) \right).
   $$
 \end{definition}
 \begin{rmk} \label{properties_cvxcj}
   We mention some properties of $g^{*}$. Let $g$ be a convex function defined on $\RR^n$.
 \begin{remlist}
 \item The convex conjugate $g^{*}$ of $g$ is also a convex function.
 \item If $g$ is increasing in each variable, then $g^{*}$ is decreasing in each variable, and vice versa.
 \item The {\bf biconjugate} $g^{**}$ of $g$ is always lower semicontinuous and $g^{**} \le g$. Moreover, $g^{**}=g$ if and only if $g$ is lower semicontinuous. \label{properties_cvxcj_3}
 \end{remlist}
 The first two statements are straightforward and see \cite[Chapter 2]{H} for the last property.
 \end{rmk}

\begin{rmk} \label{restrict_domain}
  Let $g$ be a convex function on $\RR^n_{<0} := \{ (x_1,\ldots, x_n) : x_1,\ldots, x_n < 0 \}$ which is increasing in each variable. For a positive real number $r$ and $x \in \RR^n$, we have the following inequality:
\begin{equation*}
\begin{split}
\sup_{y \in I_r^n}{ (\inp{x,y} - g(y)) } & \le \sup_{y \in \RR^n_{\le 0}}{ ( \inp{x,y} - g(y) )}   \\
&= \sup_{y \in \RR^n_{\le 0}}{ (\inp{x,y-\mathbf{r}} - g(y-\mathbf{r}) + \inp{x,\mathbf{r}} - g(y) + g(y-\mathbf{r}) )} \\
&\le \sup_{y \in I_r^n}{ (\inp{x,y} - g(y)) } + \inp{x,\mathbf{r}}
\end{split}
\end{equation*}
where $I_r = (-\infty, -r)$ is an open interval in $\RR$ and $\mathbf{r} = (r,\ldots, r) \in \RR^n$. This shows that shrinking the domain of a convex function to $I_{r}^n$ for an arbitrary small $r>0$ does not affect its Newton convex body.
\end{rmk}

 Now we are ready to define the Newton convex bodies of toric psh functions.

\begin{definition} \label{Newtonconvexbody}
Let $g$ be a convex function on $U=(-\infty, \log r)^n$. We denote by $P(g)$ the domain of the convex conjugate of $g$. In other words,
$$
x \in P(g) \iff \sup_{y \in U}(\inp{x,y} - g(y)) < + \infty.
$$
We call $P(g)$ the {\bf Newton convex body} of $g$. If $\varphi$ is a toric psh function on $D(0,r)$ and $g$ is the convex function associated to $\varphi$, we write $P(\varphi) = P(g)$ and call $P(\varphi)$ the Newton convex body of $\varphi$.
\end{definition}

 By Remark~\ref{restrict_domain}, we know that $P(\varphi)$ depends only on the germ of $\varphi$ at $0$. As we mentioned in the introduction, equisingularity of two psh functions is defined in terms of multiplier ideal sheaves.

\begin{definition}
Let $\varphi$ be a psh function defined on a complex manifold $X$. The {\bf multiplier ideal sheaf} $\mathcal{J}(\varphi)$ of $\varphi$ is the ideal sheaf of the sheaf of holomorphic functions on $X$ whose germ at each $x \in X$ is equal to
$$
\left\{ f \in \mathcal{O}_x : {\abs{f}}^2 e^{-2\varphi} \ \text{is integrable with respect to the Lebesgue measure near} \ x \right\}.
$$

\end{definition}

 Also we will say that $\varphi$ and $\psi$ are {\bf equisingular} if $\mathcal{J}(\varphi) = \mathcal{J} (\psi)$.  
 
For a monomial ideal $\mathfrak{a} \subset \CC[x_1,\ldots, x_n]$, a relation between $\mathcal{J}(\mathfrak{a})$ and the Newton polyhedron of $\mathfrak{a}$ was established by Howald \cite{H01}. This relation was extended to the case of toric psh functions and their multiplier ideals, which help us to determine whether two psh functions are equisingular or not using convex analysis only.
\begin{theorem} \label{Guemulti}
\cite[Theorem 1.13]{Gu12} \cite[Proposition 3.1]{R13} Let $\varphi$ be a toric psh function defined on $D(0,r) \subset \CC^n$. Then the multiplier ideal $\mathcal{J}(\varphi) $ is a monomial ideal and we have:
$$ z^{\alpha} \in \mathcal{J}(\varphi) \iff \alpha + \mathbf{1} \in \mathrm{int } (P(\varphi)).$$
\end{theorem}

Since we are considering psh singularities of psh germs, we may assume every toric psh function $\varphi$ is defined on a open set near the origin. Furthermore, we may restrict the domain of $\varphi$ to a closed polydisk centered at the origin with positive polyradius when we take its associated convex function. Then its associated convex function is lower semicontinuous and thus, by Remark \ref{properties_cvxcj},  $\varphi$ has the equality with its biconjugate after shdrinking the domain of $\varphi$ to a smaller open polydisk. For its conjugate, we also consider only $g^{*}$ restricted to the interior of $P(\varphi)$ so that $g^{*}$ is a continuous convex function decreasing in each variable.

\section{Newton convex bodies for analytic singularities}
 In this section, we will prove the following characterization of psh function with analytic singularities and what convex conjugates of toric psh functions with analytic singularities look like.
 \begin{proposition} \label{poles_toric}
   Let $\varphi$ be a toric psh funtion with analytic singularities on a unit polydisk $D(0,1) \subseteq \mathbf{C}^n$. Then $\varphi$ is associated to a monomial ideal with weight $c \in \RR_{>0}$, i.e. $\varphi \simeq \displaystyle\frac{c}{2} \log ({\abs{z}^{2\alpha_1}} + \cdots + {\abs{z}^{2\alpha_m}})$ near $0$ where $\alpha_1, \ldots, \alpha_m$ are multi-indices and $\simeq$ means that their difference is in $\bigo(1)$.
 \end{proposition}

 \begin{rmk} \label{pole_toricaction}
   For a toric plurisubharmonic function $\varphi$ with analytic singularities, if we write $$\varphi = \displaystyle\frac{c}{2}\log ({\lvert g_1 \rvert}^2 + \cdots +{\lvert g_r \rvert}^2) + \bigo(1)$$ near $0$, then it is certainly not the case that the value of ${\lvert g_1 \rvert}^2 + \cdots + {\lvert g_r \rvert}^2$ is independent of torus action. However, the vanishing of ${\lvert g_1 \rvert}^2 + \cdots + {\lvert g_r \rvert}^2$ is invariant under torus actions.
 \end{rmk}

 \begin{proof}
 Write
 $$
 \varphi = \frac{c}{2}\log ({\lvert g_1 \rvert}^2 + \cdots +{\lvert g_r \rvert}^2) + \bigo(1)
 $$
 near $z=0$. We use induction on $n$. Consider the case $n=1$. Suppose that $g_1, \ldots g_r$ have a common zero at $0$ with multiplicity $k$. Then we may assume $\abs{g_1}^2 + \cdots +\abs{g_r}^2$ is nonvanishing at $0$ by extracting $\abs{z}^{2k}$. If $\abs{g_1}^2 + \cdots +\abs{g_r}^2$ vanish at some point $z_0 \neq 0$, by Remark \ref{pole_toricaction}, it vanishes on the circle $\lvert z \rvert = \lvert z_0 \rvert$. By the maximum principle $\abs{g_1}^2 + \cdots +\abs{g_r}^2$ vanishes on the disk $D(0,\lvert z_0 \rvert)$, which is a contradiction. Thus $\abs{g_1}^2 + \cdots + \abs{g_r}^2$ is nowhere vanishing. In particular, it is bounded below by some positive number $C$ on some locally compact neighborhood of $0$. So, we can always write $\varphi = \frac{1}{2} \log \abs{z}^{2k} + \bigo(1)$ near $0$.

 Now, suppose $n \geq 2$. We introduce some auxiliary notations for convenience:
\begin{enumerate}
\item $H_j$ is the hyperplane in $\CC^n$ defined by $z_j$.
\item $z(i)^{\alpha(i)}$ is a monomial of $z_1 , \cdots , \widehat{z_i}, \cdots z_n$ with the multi-index exponent $\alpha(i)$.
\end{enumerate}
If the common zero set of $g_1 , \ldots , g_r$ contains all $H_j$($1 \leq j \leq n$), then similarly, one can extract $z^\alpha$ where $\alpha$ is a multi-index from all $g_1 , \ldots , g_r$ so that $\abs{g_1}^2 + \cdots + \abs{g_r}^2$ does not vanish identically on $H_j$ for each $1 \le j \le n$. So, we may assume that $\{ g_1 \ldots , g_r \}$ has no common factor which is a nontrivial monomial. Now if we restrict $\varphi$ on $H_j$, by the induction hypothesis,
 $$\restr{\varphi}{H_j} = \displaystyle\frac{c}{2} \log \left( \abs{ \restr{g_1}{H_j}}^2 + \cdots + \abs{ \restr{g_r}{H_j}}^2 \right) + \bigo(1) \simeq \displaystyle\frac{c}{2}\log \left({\lvert {z(j)}^{\alpha(j,1)} \rvert}^2 + \cdots +{\lvert {z(j)}^{\alpha(j,m_j)} \rvert}^2 \right).
 $$
 If we put
 $$
 h_j(z(j)) = \displaystyle\frac{{\lvert \restr{g_1}{H_j} \rvert}^2 + \cdots + {\lvert \restr{g_r}{H_j} \rvert}^2}{{\lvert {z(j)}^{\alpha(j,1)} \rvert}^2 + \cdots + {\lvert {z(j)}^{\alpha(j,m_j)} \rvert}^2},
 $$
 then it is nowhere vanishing, well-defined positive-valued function on $H_j$. In particular, it is bounded below by some positive number $C_j>0$. Let $C'$ be the minimal number among $C_1 , \ldots , C_n$.

 We can argue as above procedure for all $j$ and obtain the set $S$ of monomials by joining all such $z(j)^{\alpha(j, i_j)}$. Here, $1\le j \le n$ and $1\le i_j \le m_j$. We may regard such $\alpha(i,i_j)$ as a multi-index in $n$ variables inserting $0$ for $i$-th component which is the excluded index while we were restricting to the hyperplane $H_i$. So, we may re-index such messy notations by $z^{\beta_1} , \cdots , z^{\beta_l}$. Now it suffices to show the following equality:
 \begin{equation}
      \varphi = \displaystyle\frac{c}{2}\log \left( {\lvert z^{\beta_1} \rvert}^2 + \cdots + {\lvert z^{\beta_l} \rvert}^2 \right)
  \end{equation} up to $\bigo(1).$
Indeed, since every torus-invariant subvariety of $D(0,1)$ is given by an intersection of hyperplanes and $Z(g_1 , \ldots , g_r)$ does not have any codimension $1$ irreducible components, we know that $\varphi$ itself has a pole set of codimension $\geq 2$. We now observe the function $$h(z) = \displaystyle\frac{{\lvert g_1 \rvert}^2 + \cdots + {\lvert g_r \rvert}^2}{{\lvert {z}^{\beta_1} \rvert}^2 + \cdots + {\lvert {z}^{\beta_l} \rvert}^2}.$$ If it has a pole at some point $\eta$, then $\eta$ should be in some $H_j$. But on $H_j$, $h(z) \leq h_j(z)$ and $h_j(z)$ cannot blow up at $\eta$. Thus it is well-defined. Again, using similar argument with the case $n=1$, depending upon the maximum principle and Remark \ref{pole_toricaction}, we know that $h$ cannot vanish at $w$ where all $w_i$ are nonzero.
 Now it is enough to check that if $w_i=0$ for some $i$, then $h(w) \geq \displaystyle\frac{C}{n}\min_{w_j=0}(h_j(w)) \ge \frac{{C}{C^\prime}}{n}$ for some $C >0$. This follows immediately from Lemma~\ref{rearrangeineq}. Therefore, $h$ is bounded below by some positive lower bound near $0$.
 \end{proof}
 \begin{lemma} \label{rearrangeineq}
   Let $a$ and $b_i$ $(1 \le i \le n)$ be positive real numbers such that $b_1 \le \cdots \le b_n$. Then
   $$
   \frac{a}{b_1 + \cdots + b_n} \ge \frac{C}{n}\min_{1\le i \le n}{\frac{a}{b_i}}, \ \ \text{where} \ \ C={\left( \frac{b_1}{b_n} + \cdots + \frac{b_n}{b_1} \right)}^{-1}.
   $$
 \end{lemma}
 \begin{proof}
 It is straightforward from the rearrangement inequality :
 $$
 \left( \frac{b_n}{b_1} + \cdots + \frac{b_1}{b_n}\right)na \ge (b_1 + \cdots + b_n)\left( \frac{a}{b_1} + \cdots + \frac{a}{b_n}\right).
 $$
 In fact, we can take two increasing sequences by $x_i = b_i$ and $y_i = a/b_{n+1-i}$ for $1 \le i \le n$. Then $n (x_ny_n + \cdots +x_1y_1) \ge (x_1 + \ldots x_n)(y_1 + \ldots y_n)$.
 \end{proof}

 Using Proposition \ref{poles_toric}, we obtain a useful characterization for toric psh with analytic singularities.
 \begin{corollary} \label{cvxcj_analsing}
   If $\varphi$ is a toric psh with analytic singularities on $D(0,1)$, written as
 $$
 \varphi = \frac{c}{2}\log{(\abs{z}^{2b_1} + \cdots + \abs{z}^{2b_r})} + \bigo(1),
 $$
and if $g$ is the convex function associated to $\varphi$ defined on $\RR_{< 0}^n$, then $g$ is of the form $\displaystyle c \max_{1 \le i \le r} \langle b_i, x \rangle$ up to $\bigo(1)$.
 \end{corollary}

 \begin{proof}
 Since we know that $\log \max\limits_{1 \le i \le r}\abs{z}^{b_i} \le \log (\abs{z}^{b_1} + \cdots + \abs{z}^{b_r}) \le \log \max\limits_{1 \le i \le r} r \cdot \abs{z}^{b_i}$, $\varphi$ can be written as $\displaystyle\frac{c}{2}\log \max\limits_{1 \le i \le r}\abs{z}^{2b_i} + \bigo(1)$. This concludes the proof
 \end{proof}

 Using this, we can compute the Newton convex body of a toric psh function with analytic singularities. For a set of finite points $ b = \{ b_1, \ldots, b_r
\}$ in $\RR^n_{\ge 0}$, let $P(b)$ be the Minkowski addition of the convex hull of $b$ and $\RR^n_{\ge 0}$. We call $P(b)$ the closed polytope determined by $b$.

 \begin{proposition} \label{Newton_analsing}
   Let $\varphi$ be of a form as in Corollary~\ref{cvxcj_analsing}. Then $P(\varphi)$ is the closed subset in $\RR^n_{\ge 0}$ represented as $c P(b) + \RR^n_{\ge 0}$, where $b$ is the set of exponents in a representation of $\varphi$ and $P(b)$ is the convex hull of $\{ b_1 , \ldots , b_r \}$.
 \end{proposition}

 \begin{rmk}
 In Proposition~\ref{Newton_analsing}, we need not assume that $\varphi$ itself is of analytic singularities. Indeed, no conditions of $b_i$ are imposed.
 \end{rmk}

 \begin{proof}
Since $P(c \varphi) = c P(\varphi)$, we may assume that $c = 1$.
 It is just from writing condition of being in $P(\varphi)$. Denote $Q$ by $P(b) + \RR^n_{\ge 0}$. Then since each $b_i$ belongs to $P(\varphi)$, the minimality of convex hull $P(b)$ implies that $Q$ is contained in $P(\varphi)$.

 If $t \in \RR^n_{\ge 0}$ in not in $Q$, then there is a unique vector $v$ determines the distance $d(t, P(b)+ \RR^n_{\ge 0})=d(t, t+v)= \abs{v}>0$. Here, the uniqueness of $v$ follows from the convexity of set $Q$.
Also $v$ should be in $\RR^n_{\geq 0}$. This $v$ determines a unique region $A_v$ defined as:
$$ A_v = \{ y \in \RR^n \mid \inp{ -v, y } \le \inp{ -v, t+v } \}.$$
In fact, $A_v$ is the lower half-space of a supporting hyperplane of $Q$ at $t+v$ which is perpendicular to $-v$. Since $b_j$ is in $A_v$ for every $1 \le j \le r$, we know that 
$$
\inp{ t - b_j + v , -v} \ge 0.
$$

Now. let $y_k = -kv - \epsilon \in \RR_{<0}^n$ be a sequence of points in $\RR_{<0}^n$ where $k$ is a positive integer and $\epsilon$ is a small vector in $\RR_{<0}^n$. By the definition of $P(\varphi)$, the following sequence $t_k = \inp{t, y_k} - \max\limits_{1\le i \le r} \inp{b_i , y_k}$ should be bounded above. It is equivalent to the fact that $\min\limits_{1\le i \le r} \inp{t - b_i , y_k}$ is bounded above. We can estimate a lower bound of $\min\limits_{1\le i \le r} \inp{t - b_i , y_k}$ as follows:
\begin{equation*}
\begin{split}
\min\limits_{1\le i \le r} \inp{t - b_i , y_k} &= \min\limits_{1\le i \le r} \inp{t-b_i, -kv -\epsilon} \\
&= \min\limits_{1\le i \le r} \inp{t-b_i, -kv} + \inp{t-b_i, -\epsilon}  \\
&\ge \min\limits_{1\le i \le r} \left( \inp{t - b_i +v , -kv} + \inp{-v , -kv} + C \right) \\
&\ge \min\limits_{1\le i \le r} k \abs{v}^2 + C.
\end{split}
\end{equation*}

Here, $C$ is a bounded constant coming from $\inp{t-b_i, -\epsilon}$ and the last inequality comes from our observation $\inp{ t - b_j + v , -v} \ge 0$ discussed above. But this goes to $\infty$ as $k \rightarrow  \infty$ and contradicts the definition of $P(\varphi)$.
 \end{proof}

 \begin{rmk}
 The above proposition also demonstrates that the definition of the Newton convex body of a toric psh function is a generalization of the definition of Newton convex body of a monomial ideal. In fact, the Newton convex body of a monomial ideal $\mathfrak{a} = (z^{b_1}, \ldots , z^{b_m})$ where $b_1, \ldots , b_m$ are exponents of generators of $\mathfrak{a}$ is defined by $\mathrm{Conv }(b_1, \ldots , b_m) + \RR^n_{\ge 0}$. This is consistent with the Newton convex body of a toric psh function with analytic singularities determined by a monomial ideal $\mathfrak{a}$. See \cite{Gu12, H01} for more related details.
 \end{rmk}

\begin{example}[\cite{G20}] \label{guanexample}
If $\varphi = \displaystyle\max_{1\leq i \leq m} \log {\lvert z_i \rvert}^{a_i}$ defined as a germ of a toric psh function at $(\mathbf{C}^n,0)$, then by Proposition \ref{Newton_analsing}, its Newton convex body can be computed concretely, $P(\varphi)=\left( H \cap \RR^n_{\ge 0} \right)+ \RR^n_{\ge 0}$, where $H$ is the hyperplane defined by a linear equation $\displaystyle\sum_{i=1}^m \frac{x_i}{a_i} = 1$.
\end{example}

\section{Convex conjugate of analytic singularities}

In this section, we will characterize the convex conjugates of toric psh functions with analytic singularities and exhibit a relation between convergence of convex functions and convergence of their conjugate functions.

\begin{definition}  \label{H-poly}
A closed subset $P \subseteq \RR^n$ is an $\mcH$-\textit{polyhedron} if $P$ is given by the intersection of finite numbers of half-spaces. More explicitly, there exist $p$ vectors $a_1, \ldots, a_p$ and $p$ real numbers $b_1, \ldots, b_p$ such that $P$ is given by $P = \{ x \in \RR^n : \inp{a_i, x} \le b_i \text{ for all $i=1,\ldots, p$}\}$.
\end{definition}
By normal vectors in this paper, we mean \textit{outward} normal vectors. If all $a_i$ and $b_i$ can be taken to be in $\QQ^n$ and $\QQ$ respectively, $P$ is said to be \textbf{rational}.

\begin{theorem} [{\cite[Theorem 2.4.9]{Sch}}] \label{normal_conicalhull}
Let $P$ be an $\mcH$-polyhedron in $\RR^n$ and $p$ a point in the boundary of $P$. If $F_1, \ldots, F_m$ are the facet of $P$ containing $p$ and $a_1, \ldots, a_m$ are normal vectors for $F_1,\ldots, F_m$ respectively, then every normal vector $a$ of a supporting hyperplane of $P$ at $p$ is in the conical hull of $a_1,\ldots, a_m$, that is, there are nonnegative real numbers $\lambda_1,\ldots, \lambda_m$ such that
$$
a = \lambda_1 a_1 + \cdots + \lambda_m a_m. 
$$
\end{theorem}

\begin{definition} \label{V-poly}
A closed subset $P \subseteq \RR^n$ is a $\mcV$-\textit{polyhedron} if there exist a finite set of points $Y$ and a finite set of vectors $V$ such that $P$ is the sum of the convex hull of $Y$ and the conical hull of $V$, that is,
$$
P = \conv(Y) + \cone(V).
$$
\end{definition}

As in the case of $\mcH$-polyhedrons, a $\mcV$-polyhedron is said to be rational if one can take all points in $Y$ and all vectors in $V$ from $\QQ^n$.

\begin{lemma} \label{epigraph_polyhedron_maxaffine}
Let $Q$ be a rational $\mcH$-polyhedron in $\RR^n_{\le 0}$ such that $Q + \RR^n_{\le 0} \subseteq Q$ and let $g$ be a convex function with $\dom(g) = Q$, increasing in each variable. Then the followings are equivalent.
\begin{enumerate}
\item The epigraph of $g$, $\epi(g) := \{ (x',x_{n+1}) \in \RR^n \times \RR : x_{n+1} \ge g(x') \}$, is a rational $\mcH$-polyhedron.
\item There are a finite set of vectors $\{ s_1, \ldots, s_N \}$ in $\QQ^n_{\ge 0}$ and a finite set of rational numbers $\{a_1, \ldots, a_N\}$ such that
$$
g(x) = \max_{1\le i \le N}{ \left( \inp{ s_i,x } + a_i \right) }
$$
on $Q$.
\end{enumerate}

Symmetrically, if we set $P$ as a rational $\mcH$-polyhedron in $\RR^n_{\ge 0}$ such that $P + \RR^n_{\ge 0} \subseteq P$ and let $h$ be a convex function with $\dom(h) = P$, decreasing in each variable. Then the followings are equivalent.
\begin{enumerate}
\item The epigraph of $h$ is a rational $\mcH$-polyhedron.
\item There are a finite set of vectors $\{ t_1, \ldots, t_N \}$ in $\QQ^n_{\le 0}$ and a finite set of rational numbers $\{b_1, \ldots, b_N\}$ such that
$$
h(x) = \max_{1\le i \le N}{ \left( \inp{ t_i,x } + b_i \right) }
$$
on $P$.
\end{enumerate}
\end{lemma}

\begin{proof}
Suppose that $\epi(g)$ is a rational $\mcH$-polyhedron. Let $S^t x \le a$ be a minimal system of inequalities for $\epi(g)$, where $S$ is an $(n+1) \times (p+q)$ matrix $[s_1 \, \cdots \, s_{p+q}]$ with $s_i, a \in \QQ^{n+1}$. We may assume that the last $q$ vectors $s_{p+1}, \ldots ,s_{p+q}$ correspond with a minimal system of inequalities for $Q$, that is, the $(n+1)$-th coordinate of $s_k$ is nonzero if and only if $k=1,\ldots, p$. Thus we can normalize $s_1,\ldots ,s_p$ so that their $(n+1)$-th coordinates are all $-1$. Set $s_k = (s_k', -1) \in \RR^n \times \RR^1 (k=1,\ldots, p$). Now we shall prove that $g$ can be written as the form
$$
g(x') = \max_{1\le i \le p}( \inp{s_i',x'} - a_i ),
$$
where $a_i$ is an $i$-th coordinate of $a$. Then $x = (x',x_{n+1}) \in \epi(g)$ if and only if $x' \in Q$ and $x$ satisfies the \textit{nonvertical} inequaliteis in $S^t x \le a$:
\begin{align*}
\inp{s_1', x'} - a_1 & \le x_{n+1}, \\
\vdots \\
\inp{s_p', x'} - a_p & \le x_{n+1}.
\end{align*}
Equivalently, $x = (x',x_{n+1}) \in \epi(g)$ if and only if $x' \in Q$ and
\begin{equation} \tag{$\ast$}
\max_{1\le i \le p}( \inp{s_i', x'} - a_i ) \le x_{n+1}.
\end{equation}
Observing that $g(x') = \inf{ \{x_{n+1} :  (x',x_{n+1}) \in \epi(g) \} }$, we have $g(x') = \max( \inp{s_i', x'} - a_i )$. Note that every $s_i'$ should be in $\QQ^n_{\ge 0}$, because all $s_i'$ form a minimal system and $g$ is increasing.

The converse is immediate from the observation $(\ast)$.
\end{proof}
Following the same argument in the proof of Lemma~\ref{epigraph_polyhedron_maxaffine} without the assumption on $Q$ to be rational, we know that a similar statement for arbitrary $\mcH$-polyhedron holds.
\begin{rmk} \label{epigraph_rmk}
In Lemma~\ref{epigraph_polyhedron_maxaffine}, if $Q$ is a (not necessarily rational) $\mcH$-polyhedron, the followings are equivalent.
\begin{enumerate}
\item The epigraph of $g$ is an $\mcH$-polyhedron.
\item There are a finite set of vectors $\{ s_1, \ldots, s_N \}$ in $\RR^n_{\ge 0}$ and a finite set of real numbers $\{a_1, \ldots, a_N\}$ such that
$$
g(x) = \max_{1\le i \le N}{ \left( \inp{ s_i,x } + a_i \right) }
$$
on $Q$.
\end{enumerate}
\end{rmk}

\begin{theorem} \label{polyhedron_conjugate_polyhedron}
Let $g$ and $Q$ be as in Lemma~\ref{epigraph_polyhedron_maxaffine} and assume that $g$ satisfies one of the equivalent conditions in Lemma~\ref{epigraph_polyhedron_maxaffine}. If $h$ is the convex conjugate of $g$, then $\epi(h)$ is a rational $\mcV$-polyhedron.
\end{theorem}
\begin{proof}
Assume that $g$ can be written as
$$
g(x') = \max_{1\le i \le p} ( \inp{ s_i', x'} - a_i )
$$
on $Q$ with $s_i' \in \QQ^n_{\ge 0}$ and $a_i \in \QQ$. Observe that $h(s_i') = \sup_{y'}(\inp{s_i', y'} - g(y'))$ attains its supremum at any $y'$ such that $(y',g(y'))$ is on a facet $F_i$ of $\epi(g)$ which is given by the equation $\inp{(s_i',-1),x} = a_i$. Thus we have $g(y') = \inp{s_i',y'} - a_i$ for such $y'$ and thus $h(s_i') = a_i$. Observe that in general $s'$ is contained in $P = \dom(h)$ and $h(s') = k$ if and only if $\inp{(s',-1),x} = k$ is a supporting hyperplane of $\epi(g)$. For notational convenience, write $s_i = (s_i',-1)$. Let $V$ be the set of points in $\RR^{n+1}$ given by
\begin{multline*}
V = \{ (u',b) \in \RR^n_{\ge 0} \times \RR : \text{$\inp{u',x'} = b$ is a supporting hyperplane $H'$ of $Q$} \\
\text{such that $H'\cap Q$ is a facet of $Q$. } \}
\end{multline*}
We will prove
\begin{equation} \label{epih_vpoly}
\epi(h) = \conv((s_1',a_1), \ldots, (s_p',a_p)) + \cone(V \cup \{e_{n+1}\}),
\end{equation}
where $e_{n+1} = (0,\ldots,0,1) \in \RR^{n+1}$.

Let $s'$ be a point in $P$. Since $\inp{s',y'} -g(y')$ is a piecewise-affine concave function in $y'$ on $Q$, it attains the supremum, say at $y_0' \in Q$.  By the above observation, $\inp{(s',-1),x} = h(s')$ is a supporting hyperplane of $\epi(g)$ at $y_0$. If $y_0'$ is in the interior of $Q$, then $(s',-1)$ is a positive combination of the normal vectors of the \textit{nonvertical} facets of $\epi(g)$ containing $y_0$. Here, by a nonvertical facet, we mean that its normal vector has nonzero $(n+1)$-th component. Without loss of generality, suppose that $F_1, \ldots, F_m$ are the facets of $\epi(g)$ containing $y_0$. Then by Theorem~\ref{normal_conicalhull}, there exist $\lambda_1,\ldots,\lambda_m \ge 0$ such that
$$
(s',-1) = \lambda_1 s_1 + \cdots + \lambda_m s_m.
$$
Comparing the $(n+1)$-th component of both sides of this, we know that $s'$ is given by the convex combination of $s_1,\ldots, s_m$ with coefficients $\lambda_1,\ldots, \lambda_m$. Furthermore, $y_0$ satisfies the equation $\inp{s_i,x} = a_i$ for all $i = 1,\ldots, m$, the convex combination of these $m$ equations with coefficients $\lambda_1,\ldots, \lambda_m$ also holds at $y_0$. Therefore,
$$
h(s') = \inp{(s',-1),y_0} = \sum_{i=1}^{m}{ \lambda_i \inp{s_i, y_0} } = \sum_{i=1}^{m}{\lambda_i a_i}
$$
holds and thus $(s',h(s'))$ is contained in $\conv(s_1,\ldots, s_m)$.

Now assume that $y_0'$ is on the boundary of $Q$ and cannot be taken to be in the interior of $Q$. Let $u_1',\ldots, u_l'$ be normal vectors of the facets of $Q$ at $y_0'$ with $\inp{u_i', y_0'} = b_i$. Write $u_i = (u_i',0)$ for $i=1,\ldots, l$. By Theorem~\ref{normal_conicalhull} again, we obtain
\begin{equation} \label{lin_comb_at_boundary}
(s',-1) = \sum_{i=1}^{m}{\lambda_i s_i} + \sum_{j=1}^{l}{ \mu_j u_j },
\end{equation}
where $\sum_{i}{\lambda_i} = 1$ and $\mu_j \ge 0$ for all $j$. Applying $\inp{\bullet, (y_0',g(y_0'))}$ on both sides of (\ref{lin_comb_at_boundary}), we have
$$
h(s') = \sum_{i=1}^{m}{ \lambda_i h(s_i') } + \sum_{j=1}^{l}{\mu_j b_i},
$$
which implies
$$
(s',h(s')) = \sum_{i=1}^{m}{\lambda_i (s_i',a_i)} + \sum_{j=1}^{l}{ \mu_j (u_j', b_j) } \in \conv((s_1',a_1),\ldots, (s_p',a_p)) + \cone(V).
$$
This shows that $\epi(h)$ is contained in the sum of the convex hull of $(s_1',a_1), \ldots, (s_p',a_p)$ and the conical hull of $V \cup \{e_{n+1}\}$. The converse inclusion follows immediately from the definition of supporting hyperplanes.
Because the image of a $\mcV$-polyhedron under a projection is again a $\mcV$-polyhedron, we conclude that $P$ and $\epi(h)$ are $\mcV$-polyhedron. Since we can take $(s_i',a_i)$ and $(u_i',b_i)$ to be rational, $P$ and $\epi(h)$ are also rational.
\end{proof}

\begin{theorem}[\cite{Mot}, {\cite[Theorem 1.2]{Z}}] \label{H_V_equivalence}
Every $\mcH$-polyhedron is a $\mcV$-polyhedron. Also every $\mcV$-polyhedron is an $\mcH$-polyhedron.
\end{theorem}

Thanks to Theorem~\ref{H_V_equivalence}, we can drop $\mcH$ or $\mcV$ from $\mcH$-polyhedrons or $\mcV$-polyhedrons and just call them polyhedrons. Now we have the following characterization for toric psh functions with analytic singularities.

\begin{theorem} \label{analytic_sing_equivalence}
Let $\varphi$ be a toric psh function on $D(0,r) \subseteq \CC^n$ with analytic singularities and let $g$ be the convex function associated to $\varphi$. Then the domain $P$ of $g^*$ is a polyhedron such that $P + \RR^n_{\ge 0} \subseteq P$ and $(2/c)P$ is rational.
Furthermore, $g^*$ can be written as
\begin{equation} \label{max_of_affines}
g^*(y) = \frac{c}{2} \max_{1\le i \le N}{( \inp{t_i,y} + b_i )} + \bigo(1)
\end{equation}
where $t_i \in \QQ^n_{\le 0}$ and $b_i \in \QQ$.

Conversely, let $P \subseteq \RR^n_{\ge 0}$ be a polyhedron such that $(2/c)P$ is rational and $P + \RR^n_{\ge 0}\subseteq P$ and
let $h$ be a function on $P$ defined by
\begin{equation} \label{max_plus_bounded}
h(y) = \frac{c}{2} \max_{1\le i \le N}{( \inp{t_i,y} + b_i )} + v(y)
\end{equation}
where $t_i \in \QQ^n_{\le 0}$, $b_i \in \QQ$ and $v$ is a bounded function such that $h$ is convex and decreasing in each variable. Then $\varphi(z_1,\ldots, z_n) := h^*(\log{\vert z_1 \vert}, \ldots, \log{\vert z_n \vert})$ is a toric psh function with analytic singularities on $D(0,r) \subseteq \CC^n$ for some $r>0$.
\end{theorem}
\begin{proof}
This is an immediate consequence of Lemma~\ref{epigraph_polyhedron_maxaffine}, Remark~\ref{epigraph_rmk} and Theorem~\ref{polyhedron_conjugate_polyhedron}.
\end{proof}

\begin{rmk}
In the converse part of Theorem~\ref{analytic_sing_equivalence}, $r$ could be any positive real number such that
$$
(-\infty,-r)^n \subseteq \dom(h^*).
$$
\end{rmk}

For the proof of main theorem, we need to describe a relation between the convergence of a decreasing sequence of convex functions and the convergence of an increasing sequence of their conjugates provided that these functions are lower semicontinuous. Let us start with the following simple lemma.
\begin{lemma}
Let $(f_n)$ be a decreasing sequence of convex functions defined on an open subset in $\RR^n$. Then $\lim\limits_{k \to \infty} f_k$ is also convex.
\end{lemma}

\begin{proof}
We can prove the convexity directly.
\begin{equation*}
    \begin{split}
        \lim\limits_{k \to \infty} f_k(\lambda x + (1-\lambda)y) &\le
        \lambda f_n(x) + (1-\lambda) f_n(y) \\
        &\le \lambda f_m(x) + (1-\lambda) f_n(y) \\
    \end{split}
\end{equation*}
Here $m \le n$ are arbitrary positive integers. Letting $n \to \infty$ and then letting $m \to \infty$, we obtain the result.
\end{proof}

For the sake of our argument, we introduce a notion of lower semicontinuous regularization. For a family of lower semicontinuous functions $(f_\alpha)$ which is locally uniformly bounded below, its infimum $f = \inf\limits_\alpha f_\alpha$ may not be lower semicontinuous in general. Defin the \textit{lower semicontinuous regularization} of a family $(f_\alpha)$ by
$$
f^\triangle(x) = \lim\limits_{\epsilon \to 0} \inf\limits_{y \in B(x,\epsilon)} f(y) \le f(x).
$$
Then it is easy to check that $f^\triangle$ is the largest lower semicontinuous minorant of $f$. Also note that $f^\triangle(x)$ is equal to $f(x)$ whenever $f$ is lower semicontinuous at $x$. Using this notion, we are now ready to prove the following lemma.

\begin{lemma} \label{convconv}
Let $(g_m)$ be an increasing sequence of lower semicontinuous convex functions defined on $\RR^n$ converging to a convex function $g$ pointwise. Then $(g_m^*)$ is a decreasing sequence converging to $g^*$ pointwise on the relative interior of $\dom g^*$
\end{lemma}

\begin{proof}
First, we know that convex conjugate operation is order-reversing, so $(g_m^*)$ is a decreasing sequence of convex functions. Also, we know that for each $m$, $g_m^{**} =g_m$ by lower semicontinuity of $g_m$. Then using the well-known fact on conjugate functions $(\inf_\alpha f_\alpha)^*(x^*) = (\sup_\alpha f_\alpha^*)(x^*)$, we obtain
\begin{equation*}
    \begin{split}
        (\inf_m g_m^*)^*(x) &= (\sup_m g_m^{**})(x) \\
        &= (\sup_m g_m)(x) = g(x).
    \end{split}
\end{equation*}
Taking convex conjugate to both side again, we have $(\inf_m g_m^*)^{**}(x) = g^*(x)$. We observe that $\inf_m g_m^*$ is convex by the previous lemma. Applying Remark \ref{properties_cvxcj_3} to $\inf_m g_m^*$, we have the following inequalty:
$$
g^*(x) \le (\inf_m g_m^*)(x).
$$
In general, we can not say about the lower semicontinuity of $\inf_m g_m^*$. Since $g_m \le g$, we know that the sequence $(g_m^*)$ is locally uniformly bounded below by $g^*$ and we can think about the lower semicontinuous regularization of ${\inf_m g_m^*}$. By the definition of the lower semicontinuous regularization,
$$
{(\inf_m g_m^*)}^\triangle \le \inf_m g_m^*.
$$
Take the double conjugate to both sides, which preserves the order of inequalities. Note that ${(\inf_m g_m^*)}^\triangle$ is lower semicontinuous and convex, so the double conjugate of the left side is equal to itself ${(\inf_m g_m^*)}^\triangle$. For convexity of ${(\inf_m g_m^*)}^\triangle$, we refer to \cite[Proposition 2.2.2]{H}. Since ${(\inf_m g_m^*)}^{**}$ is equal to $g^*$, we have shown that
$$
{(\inf_m g_m^*)}^\triangle \le g^*.
$$
Combining this with $g^* \le \inf_m g_m^*$, we obtain
$$
{(\inf_m g_m^*)}^\triangle \le g^* \le \inf_m g_m^*.
$$

Since $\inf_m g_m^*$ is convex, it is continuous in the relative interior of $\dom g^*$. This implies that $\inf_m g_m^*$ in fact coincides with ${(\inf_m g_m^*)}^\triangle$ in the relative interior of $\dom g^*$. This concludes the proof.
\end{proof}
\section{Proof of the main theorem and Examples}
Now we are ready to prove the main theorem.

\begin{proof}[Proof of Theorem~\ref{mainthm1}]
If $(\varphi_m)$ is a decreasing sequence of toric psh functions with analytic singularities converging to $\varphi$ and $\mathcal{J}(\varphi_m) = \mathcal{J}(\varphi)$ for all $n \ge 1$, then $P:=P(\varphi_1)$ satisfies $(i)$, $(ii)$, and $(iii)$ in Theorem~\ref{mainthm1}.

Now assume that there exists a polyhedron $P$ satisfying the three conditions in the statement (2). Let $g$ be the convex function associated to $\varphi$. Then we can find a sequence of points $\{(u_i, \alpha_i)\}_{i=1}^{\infty}$ in $\QQ_{\ge 0}^{n+1} \times \QQ$ such that
\begin{equation} \label{countable_intersection}
\epi(g^*) = \frac{c}{2} \cdot \bigcap_{i=1}^\infty { H_{u_i,\alpha_i}^+ }
\end{equation}
where $H_{u,\alpha}^+$ is the closed half-space defined by $\{x \in \RR^{n+1} : \inp{u,x}\ge \alpha \}$. Indeed, let $q$ be a point  in $\RR^{n+1} \backslash \epi(g^*)$. Since $\epi(g^*)$ is a closed convex set and $d(q,\epi(g^*))>0$, there exists $(u',\alpha') \in \RR^{n+1}_{\ge 0} \times \RR$ such that $H_{u',\alpha'}$ separates $q$ and $\epi(g^*)$ strongly, that is, there exists $\epsilon > 0$ such that $q + \epsilon B(0,1) \subset \interior(H_{u',\alpha'}^-)$ and $\epi(g^*) + \epsilon B(0,1) \subset \interior(H_{u',\alpha'}^+)$. Here, $B(0,1)$ is the unit ball in $\RR^{n+1}$ and $H_{u',\alpha'}^- := \{x \in \RR^{n+1} : \inp{u',x}\le \alpha' \}$. We can choose $(u,\alpha) \in \QQ^{n+1}_{\ge 0} \times \QQ$ which is sufficiently close to $(u',\alpha')$ so that the hyperplane $H_{u,\alpha}$ also separates $q$ and $\epi(g^*)$ strongly. Enumerating all such $(u,\alpha)$ by positive integers gives (\ref{countable_intersection}).
Let $g^*_i$ be the convex function on $\RR^n$ whose epigraph is given by
$$
(P \times \RR) \cap \left( \frac{c}{2} \cdot \bigcap_{j=1}^i {H_{u_j,\alpha_j}^+ } \right).
$$
It is obvious that $g_i^*$ is increasing in each variable and lower semicontinuous.
Let $\varphi_i$ be the psh function associated to the convex conjugate of $g_i^*$. Then all $\varphi_i$ have analytic singularities by Theorem~\ref{analytic_sing_equivalence}. Furthermore, $\varphi_i$ is equisingular to $\varphi$ since the Newton convex body $P(\varphi_i)$ of $\varphi_i$ satisfies the following inclusion:
$$
P \subseteq P(\varphi_i) \subseteq P(\varphi).
$$
Note that each $g_i^*$ is of the form (\ref{max_plus_bounded}) without additional terms, we may assume that each $\varphi_i$ is defined on $D(0,r)$. Since $(g_i^*)$ is an increasing sequence of convex functions, $(\varphi_i)$ is a decreasing sequence converging to $\varphi$ on $D(0,r)$ by Lemma~\ref{convconv}.
\end{proof}

With this theorem, we can create various examples of toric psh functions which have nice approximations. For this, given a closed convex set $P \subset \RR^n_{\ge 0}$ satisfying $P + \RR^n_{\ge 0} \subset P$, we can construct a psh function $\varphi$ defined in $D(0,r) \subset \CC^n$ for some polyradius $r$ whose Newton convex body is equal to $P$ up to the boundary, that is, their interiors are the same. This implies that $\mathcal{J}(\varphi)$ and the multiplier ideal of a psh function whose Newton convex body is $P$ are equal to each other. To elaborate the statement, we introduce the following related notion.

\begin{definition} (cf. \cite{S98}, \cite{K15})
Let $\mathfrak{a}_\bullet = (\mathfrak{a}_k)$ be a graded sequence of  ideals in $\CC[z_1 , \ldots , z_n]$, i.e. $\mathfrak{a}_p \cdot \mathfrak{a}_q \subset \mathfrak{a}_{p+q}$ for any $p, q \in \ZZ_{\ge 0}$. Then a \textit{Siu psh function} associated to $\mathfrak{a}_\bullet$ is defined as
$$ \varphi = \varphi_{\mathfrak{a}_\bullet} = \log \left( \sum_{k \ge 1} \epsilon_k \abs{\mathfrak{a}_k}^{1/k} \right) $$
where $(\epsilon_k)$ is a sequence of positive real numbers chosen so that the series converges.
\end{definition}

In \cite{KS20}, it was proved that for any given convex set $P \in \RR^n_{\ge 0}$ satisfying $P + \RR^n_{\ge 0} \subset P$, there exists a graded sequence of ideals $\mathfrak{a}_\bullet$ and a Siu psh function associated to $\mathfrak{a}_\bullet$ whose Newton convex body is exactly equal to $P$ up to the boundary (see \cite[Proposition~2.9]{KS20}). As a result, for an arbitrary convex subset $P \in \RR^n_{\ge 0}$ satisfying $P + \RR^n_{\ge 0} \subset P$, we can construct a toric psh function $\varphi$ whose Newton convex body is equal to $P$ up to the boundary.

Finally, for the proof of Corollary~\ref{cor:always_approximated}, we introduce the notion of extreme points.
\begin{definition} \cite[Definition~2.1.8]{H}
Let $K$ be a convex set. A point $x$ in $K$ is called \textit{extreme} if
$$ x = \lambda_1 x_1 + \lambda_2 x_2 , \ x_1,x_2 \in K \Rightarrow x_1 = x_2 = x $$ where $\lambda_1, \lambda_2 > 0, \lambda_1 + \lambda_2 = 1$.
\end{definition}

\begin{proof}[Proof of Corollary~\ref{cor:always_approximated}]
Let $P = \{ (x,y) \in \RR^2_{\ge 0} \mid xy \ge 1 \} \medcap \RR^2_{\ge 0}$. Then there exists a Siu psh function $\varphi$ associated to a graded sequence of monomial ideals whose Newton convex body is equal to $P$ up to the boundary. We will show that for every $c>0$, $c \varphi$ satisfies the condition (2) in Theorem~\ref{mainthm1}. There are two cases of sets of lattices we need to consider. First, let $A_1, \ldots , A_N$ be lattice points in $\RR_{>0} \setminus P(c\varphi)$. Then for each $A_j$ for $1 \le j \le N$, there exists a unique point $B_j$ on $\partial P(c\varphi)$ such that the distance between $A_j$ and $B_j$ is the distance between $A_j$ and $P(c\varphi)$. Let $H_j = \{ a_j x + b_j y + c_j = 0 \}$ be the equation of tangent line of $xy = c$ at $B_j$. Then, by changing $a_j, b_j, c_j$ slightly, we can take $H_j$ having following properties.
\begin{enumerate}
    \item For each $1 \le j \le N$, $H_j$ separates $A_j$ and $P(c\varphi)$.
    \item For each $1 \le j \le N$, $H_j$ is rational.
\end{enumerate}
Secondly, there are lattice points $B'_1, \ldots,  B'_M$ on the $\partial P(c \varphi)$. Then for each $1 \le j \le M$, let $H'_j$ be the tangent line of $xy = c$ at $B'_j$. Now, if we take the polyhedron defined as
$$P = \bigcap_{j=1}^N H_j^{+} \medcap \bigcap_{j=1}^M {H'}_j^+,$$
then this $P$ exactly satisfies the condition $(2)$ in Theorem~\ref{mainthm1}. Here, $H_j^+$ and ${H'}_j^+$ are upper hyperplanes such that contains $P(c\varphi)$.
\end{proof}

\bibliographystyle{amsplain}

\quad

{\parindent0pt
\normalsize

\textsc{Jongbong An}

Department of Mathematical Sciences, Seoul National University

1 Gwanak-ro, Gwanak-gu, Seoul 08826, Republic of Korea

e-mail: ajb8406@snu.ac.kr

\quad

\textsc{Hoseob Seo}

Research Institute of Mathematics, Seoul National University

1 Gwanak-ro, Gwanak-gu, Seoul 08826, Republic of Korea

e-mail: hskoot@snu.ac.kr
}

\end{document}